\documentclass[12pt]{amsart}
\usepackage{amssymb,latexsym,amsthm}
\usepackage{graphicx}
\usepackage{color}

\newtheorem{theorem}{Theorem}[section]

\newtheorem{prop}[theorem]{Proposition}

\theoremstyle{definition}

\theoremstyle{remark}

\numberwithin{equation}{section}

\begin{document}

\baselineskip=17pt

\title[]
{Multiplicative isometries on the Smirnov class}

\author{Osamu~Hatori}
\address{Department of Mathematics, Faculty of Science, 
Niigata University, Niigata 950-2181 Japan}
\curraddr{}
\email{hatori@math.sc.niigata-u.ac.jp}

\author{Yasuo~Iida}
\address{Department of Mathematics, Iwate Medical University, Yahaba, Iwate 028-3694 Japan}
\email{yiida@iwate-med.ac.jp}

\thanks{The first author was partly 
supported by the Grants-in-Aid for Scientific 
Research, Japan Society for the Promotion of Science.}



\maketitle

\begin{abstract}
We show that $T$ is a surjective {\it multiplicative} (but not {necessarily} linear) isometry  from the Smirnov class on the open unit disk, the ball, or the polydisk onto itself, if and only if there exists a holomorphic automorphism $\Phi$
such that 
$T(f)=f\circ \Phi$ for every class element $f$ or 
$T(f)=\overline{f\circ\overline{\Phi}}$ for every class element $f$, where the automorphism 
$\Phi$ is a unitary transformation in the case of the ball and 
$\Phi(z_1,\dots ,z_n)=
(\lambda_1z_{i_1},\dots , \lambda_nz_{i_n})$ for
$|\lambda_j|=1$, $1\le j\le n$ and $(i_1, \dots , i_n)$ 
is some permutation of the integers from $1$ through $n$ in 
the case of the $n$-dimension polydisk.
\end{abstract}
\section{Introduction}
Linear isometries on the Smirnov class on 
the ball and the polydisk of 
one and several complex variables 
were studied by 
Stephenson \cite{st} and the complete representation was given. 
A long tradition of inquiry seeks representations of 
linear isometries not only on spaces of holomorphic functions but on Banach spaces and Banach algebras, and in 
 recent years linearity for given isometries are sometimes not prerequisite.
In this paper we study {\it multiplicative} isometries on the Smirnov class on the polydisk or on the open ball of $n$-complex variables for $n\ge 1$. 
The space of $n$-complex variables $z=(z_1,\dots, z_n)$ 
is denoted by 
${\mathbb C}^n$. The unit polydisk $\{z\in {\mathbb C}^n:\text{$|z_j|< 1$, $1\le j \le n$}\}$ is denoted by
$U^n$ and the distinguished boundary 
${\mathbb T}^n$ is $\{z\in {\mathbb C}^n:
\text{$|z_j|= 1$, 
$1\le j \le n$}\}$.
The unit ball $\{z\in {\mathbb C}^n:\sum_{j=1}^n|z_j|^2<1\}$ is 
denoted by $B_n$ and $S_n$ its boundary. 
 In this paper $X$ denotes the unit polydisk or 
the unit ball for $n\ge 1$ and $\Gamma$ denotes 
${\mathbb T}^n$ for $X=U^n$ and $S_n$ for $X=B_n$. 
The normalized (in the sense that 
$\sigma(\Gamma)=1$) Lebesgue measure on $\Gamma$ is 
denoted by $d\sigma$.

The Nevanlinna class $N(X)$ on $X$ is 
defined as the set of all holomorphic  functions $f$ on $X$ 
such that  

\[
\sup_{0<r<1}\int \log^+|f(rz)|d\sigma(z)<\infty
\]
or equivalently
\[
\sup_{0<r<1}\int \log(1+|f(rz)|)d\sigma(z)<\infty
\]
hold. 
It is known that $f\in N(X)$ has a finite nontangential 
limit, also denoted by $f$, almost everywhere on $\Gamma$.
The 
Smirnov class $N_{*}(X)$ is defined as the set of all 
$f\in N(X)$ 
which satisfies the equality
\[
\sup_{0<r<1}\int\log^+|f(rz)|d\sigma(z)
=
\int\log^+|f(z)|d\sigma(z)
\]
or equivalently 
\[
\sup_{0<r<1}\int \log(1+|f(rz)|)d\sigma(z)
=
\int \log (1+|f(z)|)d\sigma(z).
\]
Define a metric 
\[
d(f,g)=\int \log(1+|f(z)-g(z)|)d\sigma(z)
\]
for $f,g \in N_{*}(X)$. With the metric $d(\cdot,\cdot)$\, the Smirnov class $N_{*}(X)$ is an 
$F$-algebra. 
Recall that an $F$-algebra is a topological algebra in which the topology arises from a complete metric. For each $0<p\le\infty$, the Hardy space is denoted by $H^p(X)$ with the norm $\|\cdot\|_p$, and $H^p(X)$ is a subspace of $N_{*}(X)$, in particular, $H^{\infty}(X)$ is a dense subalgebra
of $N_{*}(X)$. The convergence on the metric is 
stronger than uniform convergence on compact subsets 
of $X$.

 In this paper we consider surjective isometries on 
the Smirnov class, we do not assume but prove the complex or conjugate linearity for these isometries.

\section{The Results}
Let $\alpha$ be a complex number with $|\alpha|=1$ 
and $\Phi$ a unitary transformation for the 
case of $X=B_n$ and $\Phi(z_1,\dots ,z_n)=
(\lambda_1z_{i_1},\dots , \lambda_nz_{i_n})$, where 
$|\lambda_j|=1$, $1\le j\le n$ and $(i_1, \dots , i_n)$ 
is some permutation of the integers from $1$ through $n$.
Then $T(f)=\alpha f\circ \Phi$ 
defines a complex-linear isometry from $N_{*}(X)$ onto 
itself (cf. \cite[Corollary 2.3]{st}). We also 
see that $T(f)=\alpha\overline{f\circ\overline{\Phi}}$ 
defines a conjugate linear isometry from 
$N_{*}(X)$ onto itself. {In particular, if $\alpha =1$, then these operators $T's$ are {\it multiplicative}.
We show that the converse holds}: multiplicative {isometries} on $N_{*}(X)$ are of one of the forms above. 
We do not assume linearity for these isometries and a crucial part of a proof is to observe that they are complex or conjugate linear. If the Smirnov class were a normed space, it could be much easier to handle by a direct application of the celebrated theorem of Mazur and Ulam for isometries between normed linear spaces \cite{mu}(cf. \cite{v}). But it is not the case.
\begin{prop}\label{main}
Let $n$ be a positive interger and let 
$X$ be either $B_n$ or $U^n$. Suppose that  
$T:N_{*}(X)\to N_{*}(X)$ is a surjective isometry. If $T$ is 2-homogeneous in the sense that
$T(2f)=2T(f)$ holds for every $f\in N_{*}(X)$, 
then either of the following formulas holds:
\begin{enumerate}
\item
$T(f)=\alpha f\circ \Phi$ for every $f \in N_{*}(X)$;
\item
$T(f)=\alpha \overline{f\circ \overline{\Phi}}$ for every
$f\in N_{*}(X)$,
\end{enumerate}
where $\alpha$ is a complex number with the unit modulus 
and for $X=B_n$, $\Phi$ is unitary transformation;
for $X=U^n$, $\Phi(z_1,\dots ,z_n)=
(\lambda_1z_{i_1},\dots , \lambda_nz_{i_n})$, where 
$|\lambda_j|=1$, $1\le j\le n$ and $(i_1, \dots , i_n)$ 
is some permutation of the integers from $1$ through $n$.
\end{prop}

\begin{proof}
A crucial part of a proof is 
to observe the linearity (complex or conjugate) of $T$, but a difficulty comes from the fact that $N_{*}(X)$ is not a normed space. Instead of directing $T$ on $N_{*}(X)$ we consider the restrictions of $T$ on certain subspaces which is a normed space; we assert first that $T(H^N(X))=H^N(X)$ for each $1\le N\le \infty$, then applying the Mazur-Ulam theorem \cite{mu}(cf. \cite{v}) concerning to isometries between {\it normed spaces} and a theorem of Ellis \cite[Theorem]{elli} on isometries between uniform algebras to observe that $T$ is complex-linear or conjugate linear. Once the linearity is established one can simply apply a theorem of Stephenson \cite{st} to prove the desired conclusion.

Let $f,g\in H^1(X)$. It requires only elementary calculation applying the 2-homogenuity of $T$ to check that 
the equation
\[
\int \log (1+|\frac{f}{2^n}-\frac{g}{2^n}|)d\sigma
=
\int \log (1+|\frac{T(f)}{2^n}-\frac{T(g)}{2^n}|)d\sigma
\]
holds. 
In a way similar to the proof of Theorem 2.1 in \cite{st} 
we see that 
$T(H^1(X))=H^1(X)$ and the restricted map $T|_{H^1(X)}$ is 
an isometry with respect to the metric induced by 
the $H^1$-norm $\|\cdot \|_1$. 
Due to the Mazur-Ulam theorem \cite{mu} $T|_{H^1(X)}$ 
is a real-linear isometry since $T(0)=0$, which is observed by just letting $f=0$ in the equation $T(2f)=2T(f)$. 

For each positive integer $N$ let the function $\theta_N$ 
on the set of the non-negative real numbers be defined as 
\begin{equation}\label{theta}
\theta_N(x)=
\begin{cases}
\frac{(-1)^N}{N+1}, \quad x=0\\
\left(\log (1+x)-\sum_{j=0}^{N-1}
\frac{(-1)^j}{j+1}x^{j+1}\right)/x^{N+1}, 
\quad x>0.
\end{cases}
\end{equation}
Note that $\sum_{j=0}^{N-1}
\frac{(-1)^j}{j+1}x^{j+1}$ corresponds to the partial sum of the Taylar series of $\log(1+x)$ for $|x|<1$. It is easy to check by a simple calculation that $\theta_N$ is continuous on $x\ge 0$ and $\left((-1)^Nx^{N+1}\theta_N(x)\right)' 
= \frac{x^N}{1+x}$ holds, hence the inequality $(-1)^N\theta _N(x)\ge 0$ holds for every $x\ge 0$. We will make use of the 
inequality $(-1)^N\theta_N(x)\ge 0$ later in applying 
Fatou theorem. 

We claim that equations
$T(H^N(X))=H^N(X)$  and   
$\|f\|_N=\|T(f)\|_N$ hold for every positive integer $N$. We prove the two equations by induction on $N$. 
We have already learnt that it is in the case for $N=1$. 
Assume that it is true that 
the two equations hold for $N=1,\dots , m$. 
Suppose that $f\in H^{m+1}(X)$. Then $f\in H^k(X)$ for every $k=1, \dots, m$ as $H^p(X)\subset H^q(X)$ for $p\ge q$. We have already learnt that $T$ is real-linear on $H^1(X)$, hence $\int \frac{|f|^k}{l}d\sigma=\int \frac{|T(f)|^k}{l}d\sigma$ holds as $\|f\|_k=\|T(f)\|_k$ for every $k=1,\dots, m$ and every positive integer $l$. Since $T$ is isometric on $N_{*}(X)$ and linear on $H^1(X)$ the equation $\int\log (1+\frac{|f|}{l})d\sigma=\int \log (1+\frac{|T(f)|}{l})d\sigma$ is observed for every positive integer $l$. Therefore 
\begin{multline*}
(-1)^m\left(
\int \log (1+\frac{|f|}{l})d\sigma 
- 
\sum_{j=0}^{m-1}\frac{(-1)^j}{j+1}\int 
\frac{|f|^{j+1}}{l}d\sigma\right) \\
=
(-1)^m\left(
\int \log (1+\frac{|T(f)|}{l})d\sigma 
- 
\sum_{j=0}^{m-1}\frac{(-1)^j}{j+1}\int 
\frac{|T(f)|^{j+1}}{l}d\sigma\right)
\end{multline*}
holds for every positive integer $l$.
It follows that 
\begin{equation}\label{moto}
\int |f|^{m+1}(-1)^m\theta_m(|f|^m/l)d\sigma 
=
\int |T(f)|^{m+1}(-1)^m\theta_m(|T(f)|^m/l)d\sigma.
\end{equation}
Letting $l\to \infty$ and applying the Lebesgue theorem on
dominated convergence to the left-hand side and 
the Fatou theorem to the right-hand side we have
\[
\int |f|^{m+1}(-1)^m\theta_m(0)d\sigma 
\ge 
\int |T(f)|^{m+1}(-1)^m\theta_m(0)d\sigma.
\]
Thus $|T(f)|^{m+1}$ is integrable and letting 
$l\to \infty$ again in the equation (\ref{moto}) 
and since $\theta_m(0)=\frac{(-1)^m}{m+1}\ne 0$ we have 
that 
\[
\int |f|^{m+1}d\sigma = \int |T(f)|^{m+1}d\sigma
\]
by the Lebesgue theorem on dominated convergence in 
both sides at this time. We conclude that $T(H^{m+1}(X))\subset H^{m+1}(X)$ and {$\|f\|_{m+1}=\|T(f)\|_{m+1}$} hold. Conversely $H^{m+1}(X)\subset T(H^{m+1}(X))$ holds in a way similar to the above.

Since $d\sigma$ is a finite measure we verify that 
\[
\lim_{m\to \infty}\|f\|_m=\|f\|_{\infty}
\]
holds for every $f\in H^{\infty}(X)$. 
Since $\|f\|_m=\|T(f)\|_m$ for every $f\in 
H^{\infty}(X)$ and $\lim_{m\to \infty}\|T(f)\|_{m}
=\|T(f)\|_{\infty}$ we see that $T(f)\in H^{\infty}(X)$ 
and $\|f\|_{\infty}=\|T(f)\|_{\infty}$ for every 
$f\in H^{\infty}(X)$. In a way similar we see 
that 
$f\in H^{\infty}(X)$ if $T(f)$ is in $H^{\infty}(X)$. 
Thus $T|_{H^{\infty}(X)}$ is a surjective isometry 
with respect to $\|\cdot\|_{\infty}$ from 
$H^{\infty}(X)$ onto itself. Since we may suppose that 
$H^{\infty}(X)$ is a uniform algebra on the maximal 
ideal space and the maximal ideal space is connected 
by the \v Silov idempotent theorem, we see that 
$T|_{H^{\infty}(X)}$ is complex-linear or 
conjugate linear (i.e. $T(\alpha f+g)=\overline{\alpha}
T(f)+T(g)$ for $f,g\in H^{\infty}(X)$) by a theorem of 
Ellis \cite[Theorem]{elli}. 

Suppose that $T|_{H^{\infty}(X)}$ is complex-linear. 
Then $T$ is complex-linear on $N_{*}(X)$ since $H^{\infty}(X)$ 
is dense in $N_{*}(X)$ and the convergence in the original 
metric is 
stronger than uniform convergence on compact subets of $X$.
It follows by Corollary 2.3 in \cite{st} that the first fomula of the conclusion is satisfied. 

Suppose that $T|_{H^{\infty}(X)}$ is a conjugate linear map. 
Then $T$ is conjugate linear on $N_{*}(X)$ as before.
Define a map $\widetilde{T}:N_{*}(X)\to N_{*}(X)$ by
\[
\widetilde{T}(f)(z_1,\dots, z_n)=
T(\overline{f(\overline{z_1},\dots, \overline{z_n})})
\]
for $f\in N_{*}(X)$. Then $\widetilde{T}$ is 
complex-linear isometry from $N_{*}(X)$ onto itself. 
Applying  Corollary 2.3 in \cite{st} to  
$\widetilde{T}$ we see that $T$ is represended by the second 
formula of the conclusion. 
\end{proof}
We say a map $T:N_{*}(X)\to N_{*}(X)$ is multiplicative 
if $T(fg)=T(f)T(g)$ for every $f,g\in N_{*}(X)$.
In this section we characterize multiplicative 
isometries from $N_{*}(X)$ onto itself.
Let $\Phi$ be a unitary transformation for $X=B_n$ and 
for $X=U^n$ $\Phi(z_1,\dots ,z_n)=
(\lambda_1z_{i_1},\dots , \lambda_nz_{i_n})$, where 
$|\lambda_j|=1$, $1\le j\le n$ and $(i_1, \dots , i_n)$ 
is some permutation of the integers from $1$ through $n$.
Then $T(f)=f\circ\Phi$ defines a complex-linear 
multiplicative 
isometry from $N_{*}(X)$ onto itself and 
$T(f)=\overline{f\circ\overline{\Phi}}$ defines a 
conjugate linear multiplicative isometry from 
$N_{*}(X)$ onto itself. We show that they are 
only muliplicative isometries from $N_{*}(X)$ onto 
itself.
\begin{theorem}
Let $T$ be a multiplicative (not necessarily linear) 
isometry from $N_{*}(X)$ onto itself. 
Then there exists a holomorphic {automorphism} $\Phi$ on $X$ such that either of the following holds:
\begin{enumerate}
\item
$T(f)=f\circ \Phi$ for every $f\in N_{*}(X)$;
\item
$T(f)=\overline{f(\overline{\Phi})}$ for every $f\in N_{*}(X)$,
\end{enumerate}
where 
$\Phi$ is unitary transformation for $X=B_n$;
$\Phi(z_1,\dots ,z_n)=(\lambda_1z_{i_1},\dots , \lambda_nz_{i_n})$ for $X=U^n$, where 
$|\lambda_j|=1$ for every $1\le j\le n$ and $(i_1, \dots , i_n)$ 
is some permutation of the integers from $1$ through $n$.
\end{theorem}
\begin{proof}
First we claim that $T(1)=1$ since $T(1)=T(1)^2$ and $T(1)$ is a holomorphic function on a connected open set $X$.
Next we show $T(2)=2$.
Put $r=\frac12$. 
Suppose that $|T(r)|>r$ on a set of positive measure on 
$\Gamma$, where $T(r)$ here is considered as the nontangential limit on $\Gamma$. Then there exists a subset $E$ of positive 
measure and $\varepsilon>0$ with
$|T(r)|\ge (1+\varepsilon )r$ on $E$. 
Since 
$\lim_{n\to \infty}\frac{\log(1+(1+\varepsilon)^nr^n)}
{\log(1+r^n)}=\infty$, there is a positive integer $n_0$ with
\[
\int_E \log(1+(1+\varepsilon)^{n_0}r^{n_0})d\sigma >
\int\log (1+r^{n_0})d\sigma. 
\]
Thus
\begin{multline*}
\int\log(1+r^{n_0})d\sigma=
\int \log (1+|T(r)|^{n_0})d\sigma
\\
\ge
\int_E\log (1+(1+\varepsilon)^{n_0}r^{n_0})
d\sigma
>\int \log (1+r^{n_0})d\sigma,
\end{multline*}
which is a contradiction proving $|T(r)|\le r$. Hence 
$|T(\frac{1}{r})|\ge \frac1r$ holds as $T(r)T(\frac{1}{r})=T(1)=1$. 
Since 
\[
\log(1+\frac{1}{r})=\int\log (1+\frac{1}{r})d\sigma=
\int\log (1+|T(\frac{1}{r})|)d\sigma
\]
we have that $|T(\frac{1}{r})|=\frac{1}{r}$ and 
 $|T(r)|=r$. 
Since $\log (1+(1-r))=d(r,1)=d(T(r),1)$
and $d(T(r),1)=\int\log (1+|1-T(r)|)d\sigma$ 
it is easy to check that $T(\frac12)=\frac12$, hence 
$T(2)=2$ for $T(2)T(\frac12)=1$. 
We have observed that $T$ is a surjective isometry 
which satisfies $T(2f)=2T(f)$ as $T$ is 
multiplicative. It follows by Proposition \ref{main} 
that
\[
T(f)=\alpha f\circ \Phi,\quad f\in N_{*}(X),
\]
or 
\[
T(f)=\alpha \overline{f(\overline{\Phi})},
\quad f\in N_{*}(X),
\]
hold for a complex number $\alpha$ and 
the holomorphic automorphism $\Phi$ 
as described in Proposition \ref{main}.  
The constant $\alpha=1$ is observed as $T(1)=1$, hence
 the conclusion holds.
\end{proof}


\end{document}